\documentclass[reqno,12pt]{article} 
\usepackage{amsmath, amssymb}
\usepackage{amsthm} 
 \newtheorem{theorem}{Theorem}[section]
 \newtheorem{corollary}[theorem]{Corollary}
 \newtheorem{lemma}[theorem]{Lemma}
 \newtheorem{proposition}[theorem]{Proposition}
 \theoremstyle{definition}
 
 \theoremstyle{remark}
 \newtheorem{remark}[theorem]{Remark}
 \newtheorem{example}{Example}
 
 \numberwithin{equation}{section}
\newcommand\on{\operatorname}

\newcommand\Div{\on{div}}

\begin{document}

%
%
%
%
%
%
%
%
%

\pagestyle{myheadings}

\title{Generalized Wintgen inequalities \\ for submanifolds of conformally flat manifolds}
%
\author{Cihan \"{O}zg\"{u}r and Adara M. Blaga}

\date{}

\maketitle

\markboth{{\small\it {\hspace{1cm} Generalized Wintgen inequalities for submanifolds of conformally flat manifolds}}}{\small\it{Generalized Wintgen inequalities for submanifolds of conformally flat manifolds \hspace{1cm}}}

\footnote{2010 \textit{Mathematics Subject Classification}.
53C40; 53C42; 53C15
}
\footnote{ 
\textit{Key words and phrases}.
Generalized Wintgen inequality; DDVV
conjecture; Conformally flat manifold; Manifold of quasi-constant curvature;
Generalized Robertson--Walker spacetime; Conformal change of the metric
}

\begin{abstract}
We obtain generalized Wintgen inequalities for submanifolds of conformally flat manifolds. As applications, we derive corresponding inequalities for submanifolds of
Riemannian manifolds of quasi-constant curvature, generalized Robertson--Walker type warped products, conformally changed real space
forms, and products of real space forms with opposite curvatures. Equality cases are also discussed. Moreover, several obstructions to the existence of minimal submanifolds are obtained from these inequalities.
\end{abstract}



\section{\textbf{Introduction}}

In 1979, Wintgen \cite{Wintgen} proved that the inequality
\begin{equation*}
G\leq \left\Vert H\right\Vert ^{2}-\left\vert G^{\perp }\right\vert
\end{equation*}%
holds for any surface $M^{2}$ in $\mathbb{E}^{4}$, where $G,$ $\left\Vert
H\right\Vert ^{2}$ and $G^{\perp }$ denote the Gauss curvature, the squared
mean curvature, and the normal curvature, respectively. He showed that the
equality holds if and only if the curvature ellipse is a circle. For
surfaces of a real space form $\widetilde{M}^{m+2}(c)$ of constant sectional
curvature $c$, Wintgen inequality was studied by Guadalupe and Rodriguez in
\cite{IVR}, and independently by Rouxel in \cite{Rouxel}. As an extension of
Wintgen's inequality, they proved that
\begin{equation*}
G\leq \left\Vert H\right\Vert ^{2}-\left\vert G^{\perp }\right\vert +c.
\end{equation*}%

The following conjecture for submanifolds in real space forms was given by
De Smet, Dillen, Verstraelen and Vrancken in 1999 \cite{DDDV}. Now, the
conjecture is known as the DDVV conjecture.

\noindent \textit{Conjecture.} Let $f:M^{n}\rightarrow \widetilde{M}^{m}(c)$
be an isometric immersion from $M^{n}$ into a real space form of constant
curvature $c$. Then
\begin{equation}
\rho \leq \left\Vert H\right\Vert ^{2}-\rho ^{\perp }+c,  \label{wint}
\end{equation}%
where $\rho $ and $\rho ^{\perp }$ denote the normalized scalar curvature
and normalized normal scalar curvature of $M^{n}$, respectively.

In \cite{DDDV}, the conjecture was proven for $n$-dimensional submanifolds, $%
n\geq \nolinebreak 2$, in codimension $2$ real space forms of constant
sectional curvature. In \cite{choi}, Choi and Lu showed that the conjecture
is true for all $3$-dimensional submanifolds in arbitrary codimension real
space forms of constant sectional curvature. When the normal connection is
flat, in \cite{Chen-96}, B.-Y. Chen proved that
\begin{equation*}
\rho \leq \left\Vert H\right\Vert ^{2}+c.
\end{equation*}

In the general form, the conjecture was proven by Lu \cite{Lu} and by Ge and
Tang \cite{GeTang} independently. DDVV inequality is also known as the
generalized Wintgen inequality now. Since then, the Wintgen inequality has
been studied by many authors for submanifolds in different ambient spaces.
See, for example, \cite%
{Aydin,Aytimur,chbl,dillen,Mihai-14,Mihai-16,Naghi,Roth,Su-17}. For the
collections of the studies about Wintgen inequality, we refer to \cite{Chen-Wint}.

Motivated by the studies of the above authors, in the present paper, we establish
new generalized Wintgen inequalities for submanifolds in conformally flat manifolds. 
As applications of our main inequality, we derive several new inequalities for 
submanifolds in Riemannian manifolds of quasi-constant curvature, generalized 
Robertson–Walker type warped products, conformally changed real space forms, 
and products of real space forms with opposite curvatures. 
We also discuss the equality cases and obtain new obstruction 
results concerning the existence of minimal submanifolds.

The paper is organized as follows: Section \ref{sect-pre} contains some
preliminaries. In Section \ref{sect-3}, we prove a generalized Wintgen
inequality for submanifolds in conformally flat manifolds. The main result
is Theorem \ref{Thm-1} that generalizes the inequality (\ref{wint}) which
was obtained in \cite{GeTang} and \cite{Lu}. Since a manifold of
quasi-constant curvature is conformally flat, as applications, in Section %
\ref{sect-4}, we provide generalized Wintgen inequalities for submanifolds
in a Riemannian manifold of quasi-constant curvature. 
Moreover, we give certain obstructions to the existence
of minimal submanifolds of conformally flat manifolds.

\section{\textbf{Preliminaries\label{sect-pre}}}

The Weyl conformal curvature tensor $\widetilde{C}$ of an $m$-dimensional
Riemannian manifold $(\widetilde{M},\widetilde{g})$ is defined by
\begin{align}
\widetilde{C}(X_{1},X_{2},X_{3},X_{4})&=\widetilde{R}%
(X_{1},X_{2},X_{3},X_{4})  \notag \\
&\hspace{-20pt}-\frac{1}{m-2}[ \widetilde{\operatorname{Ric}}(X_{2},X_{3})\widetilde{g}%
(X_{1},X_{4})-\widetilde{\operatorname{Ric}}(X_{1},X_{3})\widetilde{g}(X_{2},X_{4})  \notag
\\
&\hspace{-20pt}+\widetilde{\operatorname{Ric}}(X_{1},X_{4})\widetilde{g}(X_{2},X_{3})-%
\widetilde{\operatorname{Ric}}(X_{2},X_{4})\widetilde{g}(X_{1},X_{3})]  \notag \\
&\hspace{-20pt}+\frac{2\widetilde{\tau }}{(m-1)(m-2)}\left[ \widetilde{g}%
(X_{2},X_{3})\widetilde{g}(X_{1},X_{4})-\widetilde{g}(X_{1},X_{3})\widetilde{%
g}(X_{2},X_{4})\right] ,  \label{Weyl}
\end{align}
for any vector fields $X_{1},X_{2},X_{3},X_{4}$ tangent to $\widetilde{M}$,
where $\widetilde{R},$ $\widetilde{\operatorname{Ric}}$ and $\widetilde{\tau }$ denote the
Riemannian curvature tensor, the Ricci curvature tensor and the scalar
curvature of $(\widetilde{M},\widetilde{g})$, respectively. It is well-known
that for $m\geq 4$, the manifold is conformally flat if and only if $%
\widetilde{C}=0$.

Let $M^{n},$ $n\geq 3,$ be an $n$-dimensional submanifold of an $m$%
-dimensional conformally flat manifold $(\widetilde{M}^{m},%
\widetilde{g})$, $m\geq 4$. Using (\ref{Weyl}), the Gauss equation for the submanifold $%
M^{n}$ is
\begin{align}
R(X_{1},X_{2},X_{3},X_{4})&= \frac{1}{m-2}[\widetilde{\operatorname{Ric}}%
(X_{2},X_{3})g(X_{1},X_{4})-\widetilde{\operatorname{Ric}}(X_{1},X_{3})g(X_{2},X_{4})
\notag \\
&\hspace{-24pt}+\widetilde{\operatorname{Ric}}(X_{1},X_{4})g(X_{2},X_{3})-\widetilde{\operatorname{Ric}}%
(X_{2},X_{4})g(X_{1},X_{3})]  \notag \\
&\hspace{-24pt}-\frac{2\widetilde{\tau}}{(m-1)(m-2)}\left[%
g(X_{2},X_{3})g(X_1,X_4)-g(X_{1},X_{3})g(X_{2},X_{4})\right]  \notag \\
&\hspace{-24pt}+\widetilde{g}\left( h(X_{1},X_{4}),h(X_{2},X_{3})\right) -%
\widetilde{g}\left( h(X_{1},X_{3}),h(X_{2},X_{4})\right),
\label{eq-Gauss conform}
\end{align}
for any vector fields $X_{1},X_{2},X_{3},X_{4}$ tangent to $M^{n}$, where $R$
denotes the Riemannian curvature tensor of $(M^n,g)$, $g$ is the induced metric, and $h$ is the second
fundamental form.

The normalized scalar curvature $\rho $ \cite{DDDV} is given by
\begin{equation*}
\rho =\frac{2\tau }{n(n-1)}=\frac{2}{n(n-1)}\underset{1\leq i<j\leq n}{\sum }%
\sigma(e_{i}\wedge e_{j}),
\end{equation*}%
where $\{e_{i}\}_{1\leq i\leq n}$ is a local orthonormal frame on the
tangent bundle of the submanifold, $\sigma$ is the sectional curvature of a $%
2$-plane field, and the normalized normal scalar curvature $\rho ^{\perp }$
\cite{DDDV} is given by%
\begin{equation*}
\rho ^{\perp }=\frac{2\tau ^{\perp }}{n(n-1)}=\frac{2}{n(n-1)}\sqrt{%
\sum_{1\leq i<j\leq n}\sum_{1\leq r <s\leq m-n}\left( R^{\perp }\left(
e_{i},e_{j},\upsilon _{r},\upsilon _{s}\right) \right) ^{2}},
\end{equation*}%
where $\{\upsilon_{r}\}_{1\leq r\leq m-n}$ is a local orthonormal frame on
the normal bundle of the submanifold and $R^{\perp }$ is the curvature
tensor of the normal connection $\nabla ^{\perp }.$

\section{\textbf{Generalized Wintgen Inequality} \label{sect-3}}

In this section, following \cite{YK}, we consider the normal scalar
curvature $K_{N}$ of $M^{n}$
\begin{equation*}
K_{N}=-\frac{1}{4}\sum_{1\leq r,s\leq m-n}\operatorname{trace}\left[ A_{r},A_{s}\right]
^{2},
\end{equation*}%
where $A_{r}$, $1\leq r\leq m-n$, are the shape operators of $M^{n}$ in $(%
\widetilde{M}^{m},\widetilde{g})$, and the normalized normal scalar
curvature given by
\begin{equation*}
\rho _{N}=\frac{2}{n(n-1)}\sqrt{K_{N}}.
\end{equation*}
It is clear that
\begin{equation*}
K_{N}=-\frac{1}{2}\sum_{1\leq r<s\leq m-n}\operatorname{trace}\left[ A_{r},A_{s}\right]
^{2}=\sum_{1\leq r<s\leq m-n}\sum_{1\leq i<j\leq n}\left(g\left( \left[
A_{r},A_{s}\right] e_{i},e_{j}\right) \right)^{2}.
\end{equation*}%
In terms of the components of the second fundamental form, we have
\begin{equation}
K_{N}=\sum_{1\leq r<s\leq m-n}\sum_{1\leq i<j\leq n}\left(
\sum_{k=1}^{n}\left( h_{jk}^{r}h_{ik}^{s}-h_{ik}^{r}h_{jk}^{s}\right)
\right) ^{2}.  \label{2.1}
\end{equation}

\begin{lemma}
\label{Lem-1}Let $M^{n},n\geq 3,$ be an $n$-dimensional submanifold of an $m$%
-dimensional conformally flat manifold $(\widetilde{M}^{m},%
\widetilde{g})$, $m\geq 4$. Then
\begin{equation*}
\rho +\rho _{N}\leq \left\Vert H\right\Vert ^{2}+\frac{2}{n\left( m-2\right)
}\overset{n}{\underset{j=1}{\sum }}\widetilde{\operatorname{Ric}}(e_{j},e_{j})-\frac{2%
\widetilde{\tau }}{(m-1)(m-2)}.
\end{equation*}%
The equality case holds identically if and only if, with respect to suitable
orthonormal frames $\{e_{j}\}_{1\leq j\leq n}$ and $\{\upsilon _{r}\}_{1\leq
r\leq m-n}$, the shape operators $\{A_{r}\}_{1\leq r\leq m-n}$ of $M^{n}$ in
$(\widetilde{M}^{m},\widetilde{g})$ take the forms 
\pagebreak
\begin{equation}
A_{_{1}}=\left[
\begin{array}{ccccc}
\alpha _{1} & \beta & 0 & \cdots & 0 \\
\beta & \alpha _{1} & 0 & \cdots & 0 \\
0 & 0 & \alpha _{1} & \cdots & 0 \\
\vdots & \vdots & \vdots & \ddots & \vdots \\
0 & 0 & 0 & \cdots & \alpha _{1}%
\end{array}%
\right] ,\ \ A_{_{2}}=\left[
\begin{array}{ccccc}
\alpha _{2}+\beta & 0 & 0 & \cdots & 0 \\
0 & \alpha _{2}-\beta & 0 & \cdots & 0 \\
0 & 0 & \alpha _{2} & \cdots & 0 \\
\vdots & \vdots & \vdots & \ddots & \vdots \\
0 & 0 & 0 & 0 & \alpha _{2}%
\end{array}%
\right] ,  \label{shape}
\end{equation}%
\begin{equation*}
A_{_{3}}=\alpha _{3}I_{n},
\end{equation*}%
where $\alpha _{1},\alpha _{2},\alpha _{3},\beta $ are real functions on $%
M^{n}$, $I_{n}$ is the identity matrix and
\begin{equation*}
A_{_{4}}=\cdots =A_{_{m-n}}=0,
\end{equation*}
where $\widetilde{\operatorname{Ric}}$ and $\widetilde{\tau}$ denote the
Ricci tensor and the scalar curvature of $(\widetilde{M}^m,\widetilde{g})$, respectively.
\end{lemma}

\begin{proof}
From \cite{Mihai-14}, we know that
\begin{align}
n^{2}\left\Vert H\right\Vert ^{2}& =\sum_{r=1}^{m-n}\left(
\sum_{i=1}^{n}h_{ii}^{r}\right) ^{2}  \label{2.2} \\
& =\frac{1}{n-1}\sum_{r=1}^{m-n}\sum_{1\leq i<j\leq n}\left(
h_{ii}^{r}-h_{jj}^{r}\right) ^{2}+\frac{2n}{n-1}\sum_{r=1}^{m-n}\sum_{1\leq
i<j\leq n}h_{ii}^{r}h_{jj}^{r}.  \notag
\end{align}%
On the other hand, from \cite{Lu}, we also know that
\begin{equation*}
\sum_{r=1}^{m-n}\sum_{1\leq i<j\leq n}\left( h_{ii}^{r}-h_{jj}^{r}\right)
^{2}+2n\sum_{r=1}^{m-n}\sum_{1\leq i<j\leq n}\left( h_{ij}^{r}\right) ^{2}
\end{equation*}%
\begin{equation}
\geq 2n\sqrt{\sum_{1\leq r<s\leq m-n}\sum_{1\leq i<j\leq n}\left(
\sum_{k=1}^{n}\left( h_{jk}^{r}h_{ik}^{s}-h_{ik}^{r}h_{jk}^{s}\right)
\right) ^{2}}.  \label{2.3}
\end{equation}%
Combining (\ref{2.2}), (\ref{2.3}) and (\ref{2.1}), we have
\begin{equation}
n^{2}\left\Vert H\right\Vert ^{2}-n^{2}\rho _{N}\geq \frac{2n}{n-1}%
\sum_{r=1}^{m-n}\sum_{1\leq i<j\leq n}\left( h_{ii}^{r}h_{jj}^{r}-\left(
h_{ij}^{r}\right) ^{2}\right) .  \label{2.4}
\end{equation}%
Furthermore, from (\ref{eq-Gauss conform}), we also have
\begin{equation*}
\tau =\frac{n-1}{m-2}\overset{n}{\underset{j=1}{\sum }}\widetilde{\operatorname{Ric}}%
(e_{j},e_{j})-\frac{n(n-1)\widetilde{\tau }}{(m-1)(m-2)}+\sum_{r=1}^{m-n}%
\sum_{1\leq i<j\leq n}\left( h_{ii}^{r}h_{jj}^{r}-\left( h_{ij}^{r}\right)
^{2}\right) .
\end{equation*}%
Substituting it into (\ref{2.4}), we obtain%
\begin{equation*}
n^{2}\left\Vert H\right\Vert ^{2}-n^{2}\rho _{N}\geq \frac{2n}{
n-1}\left( \tau -\frac{n-1}{m-2}\overset{n}{\underset{j=1}{\sum }}%
\widetilde{\operatorname{Ric}}(e_{j},e_{j})+\frac{n(n-1)\widetilde{\tau }}{(m-1)(m-2)}%
\right) ,
\end{equation*}%
which implies%
\begin{equation*}
\left\Vert H\right\Vert ^{2}-\rho _{N}\geq \rho -\frac{2}{n\left( m-2\right)
}\overset{n}{\underset{j=1}{\sum }}\widetilde{\operatorname{Ric}}(e_{j},e_{j})+\frac{2%
\widetilde{\tau }}{(m-1)(m-2)}.
\end{equation*}

The equality case of the inequality holds if and only if the shape operators
are in the above forms with respect to suitable frames.
\end{proof}

It is known that a submanifold is minimal (cf. \cite{Ch73}) if the mean
curvature $H=0$.

From the previous lemma, we deduce:

\begin{corollary}
Let $M^{n},n\geq 3,$ be an $n$-dimensional minimal submanifold of an $m$%
-dimensional conformally flat manifold $(\widetilde{M}^{m},%
\widetilde{g})$, $m\geq 4$. Then
\begin{equation*}
\rho+\rho _{N} \leq \frac{2}{n\left( m-2\right) } \overset{n}{\underset{j=1}{%
\sum }}\widetilde{\operatorname{Ric}}(e_{j},e_{j}) -\frac{2\widetilde{\tau }}{(m-1)(m-2)}.
\end{equation*}
\end{corollary}

Now we can prove the following result:
\begin{theorem}\label{teor1}
\label{Thm-1}Let $M^{n},n\geq 3,$ be an $n$-dimensional submanifold of an $m$%
-dimensional conformally flat manifold $(\widetilde{M}^{m},%
\widetilde{g})$, $m\geq 4$. Then
\begin{equation*}
\rho +\rho ^{\perp }\leq \left\Vert H\right\Vert ^{2}+\frac{2}{n\left(
m-2\right) }\overset{n}{\underset{j=1}{\sum }}\widetilde{\operatorname{Ric}}(e_{j},e_{j})-%
\frac{2\widetilde{\tau }}{(m-1)(m-2)}.
\end{equation*}%
The equality case holds identically if and only if, with respect to suitable
orthonormal frames $\{e_{j}\}_{1\leq j\leq n}$ and $\{\upsilon _{r}\}_{1\leq
r\leq m-n}$, the shape operators $\{A_{r}\}_{1\leq r\leq m-n}$ of $M^{n}$ in
$(\widetilde{M}^{m},\widetilde{g})$ take the forms $(\ref{shape})$, where $\widetilde{\operatorname{Ric}}$ and $\widetilde{\tau}$ denote the
Ricci tensor and the scalar curvature of $(\widetilde{M}^m,\widetilde{g})$, respectively.
\end{theorem}

\begin{proof}
Since \(\widetilde M^m\) is conformally flat, we have
\[
\widetilde R(e_i,e_j,e_{n+r},e_{n+s})=0
\]
for tangent vectors \(e_i,e_j\) and normal vectors \(e_{n+r},e_{n+s}\).
Thus, the Ricci equation gives
\begin{equation*}
g\left( R^{\perp }(e_{i},e_{j})e_{n+r},e_{n+s}\right) =g\left( \left[
A_{r},A_{s}\right] e_{i},e_{j}\right) ,
\end{equation*}%
for all $1\leq i,j\leq n$ and $1\leq r,s\leq m-n.$ Then we find
\pagebreak
\begin{align*}
(\tau ^{\perp })^{2}& =\overset{}{\underset{1\leq r<s\leq m-n}{\sum }}%
\overset{}{\underset{1\leq i<j\leq n}{\sum }}\left( g\left( R^{\perp
}(e_{i},e_{j})e_{n+r},e_{n+s}\right) \right) ^{2} \\
& =\overset{}{\underset{1\leq r<s\leq m-n}{\sum }}\overset{}{\underset{1\leq
i<j\leq n}{\sum }}\left( g\left( \left[ A_{r},A_{s}\right]
e_{i},e_{j}\right) \right) ^{2}=K_{N}=\frac{n^{2}(n-1)^{2}}{4}\rho _{N}^{2}
\end{align*}%
and
\begin{equation*}
\left( \rho ^{\perp }\right) ^{2}=\rho _{N}^{2}.
\end{equation*}%
Since both  $\rho ^{\perp}$ and $\rho _{N}$ are nonnegative, this gives
\begin{equation*}
\rho ^{\perp }=\rho _{N}.
\end{equation*}%
Then, using Lemma \ref{Lem-1}, we have the result.
\end{proof}

A direct consequence of the above theorem is the following:

\begin{corollary}
\label{Cor-1}Under the hypotheses of Theorem \ref{teor1}, if $M^{n}$ is a minimal submanifold, then
\begin{equation}
\rho +\rho ^{\perp }\leq \frac{2}{n\left( m-2\right) }\overset{n}{\underset{%
j=1}{\sum }}\widetilde{\operatorname{Ric}}(e_{j},e_{j})-\frac{2\widetilde{\tau }}{(m-1)(m-2)%
}. \label{cor34}
\end{equation}
\end{corollary}

As an immediate consequence of Corollary \ref{Cor-1}, we obtain the
following obstructions to the existence of minimal submanifolds:

\begin{corollary}
Under the hypotheses of Theorem \ref{teor1}, if at a point $x\in M^n$,
\begin{equation*}
\rho +\rho ^{\perp }>\frac{2}{n(m-2)}\sum_{j=1}^{n}\widetilde{\operatorname{Ric}}%
(e_{j},e_{j})-\frac{2\widetilde{\tau }}{(m-1)(m-2)},
\end{equation*}%
then $M^{n}$ is not minimal at $x$.
\end{corollary}

\begin{corollary}
Under the hypotheses of Theorem \ref{teor1}, we suppose that along $M^n$,
\begin{equation*}
\frac{2}{n(m-2)}\sum_{j=1}^{n}\widetilde{\operatorname{Ric}}(e_{j},e_{j})-\frac{2\widetilde{%
\tau }}{(m-1)(m-2)}\leq 0.
\end{equation*}%
If
\begin{equation*}
\rho +\rho ^{\perp }>0
\end{equation*}%
at some point of $M^n$, then $M^{n}$ is not minimal at that point. In
particular, there exists no minimal immersion satisfying the above ambient
curvature condition and $\rho +\rho ^{\perp }>0$ everywhere.
\end{corollary}

When the submanifold is compact, we have the following result:

\begin{theorem}
Let $M^{n}$, $n\geq 3$, be a compact $n$-dimensional submanifold of an $m$%
-dimensional conformally flat manifold $(\widetilde{M}^{m},%
\widetilde{g})$, $m\geq 4$. If
\begin{equation*}
\int_{M^{n}}\left( \rho +\rho ^{\perp }\right) \,dV>\int_{M^{n}}\left( \frac{2}{%
n(m-2)}\sum_{j=1}^{n}\widetilde{\operatorname{Ric}}(e_{j},e_{j})-\frac{2\widetilde{\tau }}{%
(m-1)(m-2)}\right) \,dV,
\end{equation*}%
then the given submanifold is not minimal.
\end{theorem}

\begin{proof}
Suppose that $M^{n}$ is minimal. Then, by Corollary \ref{Cor-1}, we have the inequality (\ref{cor34})
at every point of $M^n$. Integrating this inequality over the compact manifold
$M^n$, we get
\begin{equation*}
\int_{M^{n}}\left( \rho +\rho ^{\perp }\right) \,dV\leq \int_{M^{n}}\left( \frac{2}{%
n(m-2)}\sum_{j=1}^{n}\widetilde{\operatorname{Ric}}(e_{j},e_{j})-\frac{2\widetilde{\tau }}{%
(m-1)(m-2)}\right) \,dV,
\end{equation*}%
which contradicts the hypothesis. Therefore $M^{n}$ cannot be minimal.
\end{proof}

\section{\textbf{Some Applications} \label{sect-4}}

Let $(\widetilde{M},\widetilde{g})$ be a Riemannian manifold. If the
curvature tensor $\widetilde{R}$ of $\widetilde{M}$ satisfies
\begin{align}
\widetilde{R}(X_{1},X_{2},X_{3},X_{4})& =p\left[ \widetilde{g}(X_{2},X_{3})%
\widetilde{g}(X_{1},X_{4})-\widetilde{g}(X_{1},X_{3})\widetilde{g}%
(X_{2},X_{4})\right]  \notag \\
& \hspace{12pt}+q\left[ \widetilde{g}(X_{1},X_{4})\psi (X_{2})\psi (X_{3})-%
\widetilde{g}(X_{1},X_{3})\psi (X_{2})\psi (X_{4})\right.  \notag \\
& \hspace{12pt}\left. +\widetilde{g}(X_{2},X_{3})\psi (X_{1})\psi (X_{4})-%
\widetilde{g}(X_{2},X_{4})\psi (X_{1})\psi (X_{3})\right] ,  \label{quasi}
\end{align}%
for any vector fields $X_{1},X_{2},X_{3},X_{4}$ tangent to $\widetilde{M}$,
then it is called a \textit{manifold of quasi-constant curvature} \cite%
{Chen-Yano} with associated functions $p$ and $q$, where $p,q$ are real
smooth functions on $\widetilde{M}$, and $\psi $ is a $1$-form defined by
\begin{equation}
\psi (X_{1})=\widetilde{g}(X_{1},V),  \label{T(X)}
\end{equation}%
where $V$ is a unit vector field. It is clear that the manifold $\widetilde{M%
}$ is conformally flat if it is a manifold of quasi-constant curvature; $%
\widetilde{M}$ is a space of constant sectional curvature if $q=0$. We
mention that in \cite{Ozgur-2011} and \cite{Ozgur-25}, B.-Y. Chen
inequalities were obtained for submanifolds of a Riemannian manifold of
quasi-constant curvature and a conformally flat manifold, respectively. Furthermore, inequalities for the generalized normalized
$\delta$-Casorati curvatures for submanifolds in conformally flat manifolds were found in \cite{Ozgur-2026}.

Theorem \ref{Thm-1} gives us the following inequality which is a special case of the inequality given in \cite{Su-17}:

\begin{corollary}\label{Cor-22}
Let $M^{n},n\geq 3,$ be an $n$-dimensional submanifold of an $m$-dimensional
manifold of quasi-constant curvature $(\widetilde{M}^{m},\widetilde{g}),$ $%
m\geq 4.$ Then
\begin{equation}
\rho+\rho ^{\perp }\leq \left\Vert H\right\Vert ^{2}+p+\frac{2q}{n}%
\left\Vert V^{\top}\right\Vert ^{2}.  \label{quasi-cor-2}
\end{equation}
\end{corollary}

\begin{proof}
From (\ref{quasi}),
we have
\begin{equation}
\widetilde{\operatorname{Ric}}(X_{1},X_{2})=\left[ \left( m-1\right) p+q\right] \widetilde{g%
}(X_{1},X_{2})+\left( m-2\right) q\psi(X_{1})\psi(X_{2}) ,  \label{RicN}
\end{equation}%
for any vector fields $X_{1},X_{2}$ tangent to $\widetilde{M}^m$, and
\begin{equation}
2\widetilde{\tau }=m\left[ \left( m-1\right) p+q\right] +\left( m-2\right) q.
\label{tauN}
\end{equation}%
Substituting (\ref{RicN}) and (\ref{tauN}) into Theorem \ref{Thm-1}, we
obtain (\ref{quasi-cor-2}).
\end{proof}

As a consequence, we deduce:

\begin{corollary}
\label{Cor-2}
Under the hypotheses of Corollary \ref{Cor-22},

(i) if $V$ is tangent to $M^{n}$, then
\begin{equation*}
\rho+\rho ^{\perp }\leq \left\Vert H\right\Vert ^{2}+p+\frac{2q}{n};
\end{equation*}

(ii) if $V$ is normal to $M^{n}$, then%
\begin{equation*}
\rho+\rho ^{\perp }\leq \left\Vert H\right\Vert ^{2}+p.
\end{equation*}
\end{corollary}

\begin{remark}
If $q=0$, then $(\widetilde{M}^{m},\widetilde{g})$ reduces to a space of
constant sectional curvature. Then, taking $q=0$ in Corollary \ref{Cor-22}
and Corollary \ref{Cor-2}, we get the generalized Wintgen inequality (\ref%
{wint}) which was obtained in \cite{GeTang} and \cite{Lu}.
\end{remark}

Using Corollary \ref{Cor-22}, we have the following non-existence results for
minimal submanifolds:

\begin{corollary}
\label{Cor3}
Under the hypotheses of Corollary \ref{Cor-22}, if
at a point $x\in M^{n}$,
\begin{equation*}
\rho +\rho ^{\perp }>p+\frac{2q}{n}\Vert V^{\top }\Vert ^{2},
\end{equation*}%
then $M^{n}$ is not minimal at $x$.

In particular, if
\begin{equation*}
p+\frac{2q}{n}\Vert V^{\top }\Vert ^{2}\leq 0
\end{equation*}%
and
\begin{equation*}
\rho +\rho ^{\perp }>0
\end{equation*}%
at some point of $M^{n}$, then $M^{n}$ is not minimal at that point.
\end{corollary}

\begin{corollary}
Under the hypotheses of Corollary \ref{Cor-22},

(i) if $V$ is tangent to $M^{n}$ and
\begin{equation*}
\rho +\rho ^{\perp }>p+\frac{2q}{n}
\end{equation*}%
at a point $x\in M^{n}$, then $M^{n}$ is not minimal at $x$;

(ii) if $V$ is normal to $M^{n}$ and
\begin{equation*}
\rho +\rho ^{\perp }>p
\end{equation*}%
at a point $x\in M^{n}$, then $M^{n}$ is not minimal at $x$.
\end{corollary}

\begin{example}
Let $\left( M^{m-1}(c),g\right) $ be a real
space form of dimension $\left( m-1\right) $ and constant sectional
curvature $c$. Let $I\subseteq \mathbb{R}$ be an open interval and let $%
f:I\rightarrow \mathbb{R}$ be a nowhere-vanishing smooth function. Then the
warped product $\widetilde{M}^{m}=I\times _{f}M^{m-1}(c)$ endowed with the
metric $\widetilde{g}=dt^{2}+f^{2}g$ is a generalized Robertson--Walker
spacetime (see \cite{Al}). We denote by $\frac{\partial }{\partial t}$ the
unit vector field tangent to the factor $I$. Then the Riemannian curvature
tensor of $\left( \widetilde{M}^{m},\widetilde{g}\right) $ is given by%
\begin{align}
\widetilde{R}(X_{1},X_{2},X_{3},X_{4})& =\left( \frac{c-\left( f^{\prime
}\right) ^{2}}{f^{2}}\right) \left[ \widetilde{g}(X_{2},X_{3})\widetilde{g}%
(X_{1},X_{4})-\widetilde{g}(X_{1},X_{3})\widetilde{g}(X_{2},X_{4})\right]
\notag \\
& \hspace{-20pt}-\left( \frac{f^{\prime \prime }}{f}+\frac{c-\left(
f^{\prime }\right) ^{2}}{f^{2}}\right) \left[ -\widetilde{g}(X_{1},X_{3})%
\widetilde{g}\left( X_{2},\frac{\partial }{\partial t}\right) \widetilde{g}%
\left( X_{4},\frac{\partial }{\partial t}\right) \right.   \notag \\
& \hspace{-20pt}+\widetilde{g}(X_{2},X_{3})\widetilde{g}\left( X_{1},\frac{%
\partial }{\partial t}\right) \widetilde{g}\left( X_{4},\frac{\partial }{%
\partial t}\right)   \notag \\
& \hspace{-20pt}+\widetilde{g}(X_{1},X_{4})\widetilde{g}\left( X_{2},\frac{%
\partial }{\partial t}\right) \widetilde{g}\left( X_{3},\frac{\partial }{%
\partial t}\right)   \notag \\
& \left. \hspace{-20pt}-\widetilde{g}(X_{2},X_{4})\widetilde{g}\left( X_{1},%
\frac{\partial }{\partial t}\right) \widetilde{g}\left( X_{3},\frac{\partial
}{\partial t}\right) \right] ,  \label{warped}
\end{align}%
(see \cite{Lawn} and \cite{Roth}). Hence, $\widetilde{M}^{m}=I\times
_{f}M^{m-1}(c)$ is a manifold of quasi-constant curvature with associated
functions $\left( \frac{c-\left( f^{\prime }\right) ^{2}}{f^{2}}\right) $
and $-\left( \frac{f^{\prime \prime }}{f}+\frac{c-\left( f^{\prime }\right)
^{2}}{f^{2}}\right) $, respectively, a manifold of constant curvature
provided that $\frac{f^{\prime \prime }}{f}=-\frac{c-\left( f^{\prime
}\right) ^{2}}{f^{2}}$ is a constant. Therefore, we notice that, an $m$-dimensional generalized Robertson--Walker
spacetime $\widetilde{M}^{m}=I\times _{f}M^{m-1}(c)$ is a space of constant
curvature $-k\in \mathbb R$ if and only if $\frac{f^{\prime \prime }}{f}=-\frac{c-\left(
f^{\prime }\right) ^{2}}{f^{2}}=k$, equivalently, if and only if
one of the following assertions holds:

(i) $f(t)=at+b$, $a,b\in \mathbb{R}$ and $c=a^2\geq 0$; in particular, $%
M^{m-1}$ is a flat manifold if and only if $f$ is constant, in which case, $%
\widetilde{M}^{m}$ is a direct product manifold;

(ii) $f(t)=ae^{\sqrt{k}t}+be^{-\sqrt{k}t}$, $a,b\in\mathbb{R}$, $k>0$ and $c=-4abk$;

(iii) $f(t)=a\cos (\sqrt{-k}t)+b\sin (\sqrt{-k}t)$, $a,b\in \mathbb{R}$, $k<0
$ and $c=-k(a^{2}+b^{2})\geq 0$.
\end{example}

\begin{remark} In general, a spacetime conventionally means a Lorentzian manifold.
But in this section, we use the Riemannian analogue of a generalized Robertson–Walker spacetime.
\end{remark}

Now, by Corollary \ref{Cor-22}, we deduce the following result given in
\cite{Roth}:

\begin{proposition}\label{P-1}
Let $N^{n}$, $n\geq 3$, be an $n$-dimensional submanifold of an $m$%
-dimensional generalized Robertson--Walker spacetime $\widetilde{M}%
^{m}=I\times _{f}M^{m-1}(c).$

(i) If $\frac{\partial }{\partial t}$ is tangent to $N^{n}$, then
\begin{equation*}
\rho+\rho ^{\perp }\leq \left\Vert H\right\Vert ^{2}+\left( \frac{c-\left(
f^{\prime }\right) ^{2}}{f^{2}}\right) \left( 1-\frac{2}{n}\right) -\frac{2}{%
n}\frac{f^{\prime \prime }}{f}.
\end{equation*}

(ii) If $\frac{\partial }{\partial t}$ is normal to $N^{n}$, then%
\begin{equation*}
\rho +\rho ^{\perp }\leq \left\Vert H\right\Vert ^{2}+\frac{c-\left(
f^{\prime }\right) ^{2}}{f^{2}}.
\end{equation*}
\end{proposition}

Using the above proposition, we have the following non-existence results for
minimal submanifolds:

\begin{corollary}
Under the hypotheses of Proposition \ref{P-1}, if at a point $x\in N^{n}$,
\begin{equation*}
\rho +\rho ^{\perp }>\frac{c-\left(f^{\prime }\right) ^{2}}{f^{2}}-\frac{2}{n}\left( \frac{%
f^{\prime \prime }}{f}+\frac{c-\left(f^{\prime }\right) ^{2}}{f^{2}}\right) \left\Vert
\left( \frac{\partial }{\partial t}\right) ^{\top }\right\Vert ^{2},
\end{equation*}%
then $N^{n}$ is not minimal at $x$.
\end{corollary}

\begin{corollary}
Under the hypotheses of Proposition \ref{P-1}, we suppose that along $N^{n}$,
\[
\frac{c-(f')^2}{f^2}
-\frac{2}{n}
\left(
\frac{f''}{f}+\frac{c-(f')^2}{f^2}
\right)
\left\|
\left(\frac{\partial}{\partial t}\right)^\top
\right\|^2
\leq 0.
\]
If
\[
\rho+\rho^\perp>0
\]
at some point of $N^{n}$, then $N^{n}$ is not minimal at
that point. In particular, there exists no minimal immersion satisfying the
above curvature condition and $\rho+\rho^\perp>0$ everywhere.
\end{corollary}

\begin{example}
Let \((\overline{M}^{m},\overline{g})\), \(m\geq 4\), be a Riemannian
space form of constant sectional curvature \(c\). Let
\[
\widetilde{g}=e^{2u}\overline{g},
\]
where \(u\in C^{\infty}(\overline{M}^{m})\). Denote by
\(\overline{\nabla}\), \(\operatorname{grad}\), \(\overline{\Delta}\) and
\(\overline{\operatorname{Hess}}\) the Levi-Civita connection, the gradient, the
Laplacian and the Hessian with respect to \(\overline{g}\). We use the sign convention
\[
\overline{\Delta}u={\Div}(\operatorname{grad}u).
\]
Then the curvature tensor $\widetilde{R}$, the Ricci tensor $\widetilde{\operatorname{Ric}}$%
, and the scalar curvature $2\widetilde{\tau }$ of $(\overline{M}^{m},%
\widetilde{g})$ are given by
\begin{equation*}
\begin{split}
\widetilde{R}(X_{1},X_{2})X_{3}& =\overline{R}(X_{1},X_{2})X_{3}+(\overline{%
\nabla }_{X_{1}}B)(X_{2},X_{3})-(\overline{\nabla }_{X_{2}}B)(X_{1},X_{3}) \\
& \quad +B(X_{1},B(X_{2},X_{3}))-B(X_{2},B(X_{1},X_{3})),
\end{split}%
\end{equation*}%
where
\begin{equation*}
B(X_{1},X_{2})=X_{1}(u)X_{2}+X_{2}(u)X_{1}-\overline{g}(X_{1},X_{2})\operatorname{grad}u,
\end{equation*}%
\begin{equation}
\begin{split}
\widetilde{\operatorname{Ric}}(X_{1},X_{2})& =(m-1)c\,\overline{g}(X_{1},X_{2}) \\
& \quad -(m-2)\left\{ \overline{\operatorname{Hess}}\, u(X_{1},X_{2})-X_{1}(u)X_{2}(u)%
\right\}  \\
& \quad -\left\{ \overline{\Delta }u+(m-2)\Vert \operatorname{grad} u\Vert
^{2}\right\} \overline{g}(X_{1},X_{2}),
\end{split}
\label{Ric-last}
\end{equation}%
and
\begin{equation}
2\widetilde{\tau }=(m-1)e^{-2u}\left\{ mc-2\overline{\Delta }u-(m-2)\Vert
\operatorname{grad} u\Vert ^{2}\right\} ,  \label{taulast}
\end{equation}%
respectively, (see \cite{Besse}).
\end{example}

Since a real space form is conformally flat and the Weyl tensor is invariant under conformal changes of the metric, $(\overline{M}^{m},\widetilde{g})$ is also conformally flat (see \cite{Ch84}).
Hence, by using (\ref{Ric-last}) and (\ref{taulast}) in Theorem \ref{Thm-1}, we can state the following result:

\begin{proposition}\label{Cor-last}
Let $(M^m,g)$, $m\geq 4$, be a Riemannian space form of constant
sectional curvature $c$, and let
\[
\tilde{g}=e^{2u}g,
\]
where $u\in C^{\infty}(M^m)$. Let $N^n$, $n\geq 3$, be an
$n$-dimensional submanifold of $(M^m,\tilde{g})$. Then
\[
\rho+\rho^{\perp}\leq \|H\|^2
+e^{-2u}\left[
c+\frac{2}{n}\left\{\|\operatorname{grad}^{\top}u\|^2
-\Delta^{\top}u\right\}
-\|\operatorname{grad}u\|^2
\right],
\]
where $\operatorname{grad}^{\top}u$ denotes the tangential component of
$\operatorname{grad}u$ along $N^n$, and
\[
\Delta^{\top}u=\sum_{j=1}^{n}\operatorname{Hess}u(E_j,E_j),
\]
with respect to a local orthonormal tangent frame
$\{E_j\}_{1\leq j\leq n}$ on $N^n$ for the metric $g$. 
\end{proposition}

Let us remark that if $\{E_j\}_{1\leq j\leq n}$ is a local orthonormal tangent frame on $N^n$ for the metric $g$, then
$\{e_j=e^{-u} E_j\}_{1\leq j\leq n}$ is a local orthonormal tangent frame on $N^n$
with respect to $\tilde{g}$. 
Moreover, all gradients, norms, Hessians, and Laplacians appearing in the
above inequality are computed with respect to $g$.

Using the above proposition, we have the following non-existence results for
minimal submanifolds:

\begin{corollary}
Under the hypotheses of Proposition \ref{Cor-last}, if at a point $x\in N^n$,
\[
\rho+\rho^\perp >
e^{-2u}\left[
c+\frac{2}{n}\left\{
\|\operatorname{grad}^\top u\|^2-\Delta^\top u
\right\}
-\|\operatorname{grad}u\|^2
\right],
\]
then $N^n$ is not minimal at $x$.
\end{corollary}

\begin{corollary}
Under the hypotheses of Proposition \ref{Cor-last}, we suppose that along $N^n$,
\[
c+\frac{2}{n}\left\{
\|\operatorname{grad}^\top u\|^2-\Delta^\top u
\right\}
-\|\operatorname{grad}u\|^2
\leq 0.
\]
If
\[
\rho+\rho^\perp>0
\]
at some point of $N^n$, then $N^n$ is not minimal at that point. In particular,
there exists no minimal immersion satisfying the above curvature condition
and $\rho+\rho^\perp>0$ everywhere.
\end{corollary}

\begin{example}
Let $p,q\ge 2$, $c\neq0$ and consider the Riemannian product
\begin{equation*}
\widetilde{M}^{p+q}=M_1^p(c)\times M_2^q(-c),
\end{equation*}
endowed with the product metric $\widetilde{g}=g_1+g_2$, where $M_1^p(c)$ and
$M_2^q(-c)$ are real space forms of constant sectional curvatures $c$ and $%
-c $, respectively. Such a product (of two space forms
with opposite curvatures) is conformally flat, and both factors have
dimension at least $2$ (see \cite{Laf}).

Since $M_1^p(c)$ and $M_2^q(-c)$ are real space forms of constant sectional 
curvatures $c$ and $-c$, respectively, their Ricci tensors are

\[
\operatorname{Ric}_1=(p-1)c\,g_1,\qquad
\operatorname{Ric}_2=-(q-1)c\,g_2,
\]
respectively. For the product manifold, the Ricci tensor is given by:
\[
\widetilde{\operatorname{Ric}}(X_{1},X_{2})
=
\operatorname{Ric}_1((\pi_1)_*X_{1},(\pi_1)_*X_{2})
+
\operatorname{Ric}_2((\pi_2)_*X_{1},(\pi_2)_*X_{2}),
\]
where $\pi_1$ and $\pi_2$ are the projections. Therefore,
\begin{equation}
\widetilde{\operatorname{Ric}}(X_{1},X_{2})
=
(p-1)c\,g_1((\pi_1)_*X_{1},(\pi_1)_*X_{2})
-
(q-1)c\,g_2((\pi_2)_*X_{1},(\pi_2)_*X_{2}).
\label{Ricpro}
\end{equation}
\pagebreak

\noindent The scalar curvature $2\widetilde{\tau }$ is found as
\begin{equation}
2\widetilde{\tau }=c\,(p-q)(p+q-1). \label{scalpro}
\end{equation}
\end{example}

Let $N^n$, $n\geq 3$, be an $n$-dimensional submanifold of
$\widetilde{M}^{p+q}$. For a local orthonormal tangent frame
$\{e_{j}\}_{1\leq j\leq n}$ on $N^n$, we put
\[
\lambda=\sum_{j=1}^n\|(\pi_1)_*e_j\|^2,
\]
where $\pi_1:\widetilde{M}^{p+q}\to M_1^p(c)$ denotes the projection on the first component.

Now using Theorem \ref{Thm-1}, we have the following result:

\begin{proposition}\label{P-2}
Let $N^n$, $n\geq 3$, be an $n$-dimensional submanifold of the
Riemannian product
\[
\widetilde{M}^{p+q}=M_1^p(c)\times M_2^q(-c),
\]
where $p,q\geq 2$, $c\neq 0$, and $M_1^p(c)$ and $M_2^q(-c)$ are real
space forms of constant sectional curvatures $c$ and $-c$, respectively.
Then
\[
\rho+\rho^\perp
\leq
\|H\|^2
+
\frac{2c}{n(p+q-2)}
\left[(p-1)\lambda-(q-1)(n-\lambda)\right]
-
\frac{c(p-q)}{p+q-2}.
\]
\end{proposition}

\begin{proof}
Since
\[
\widetilde{M}^{p+q}=M_1^p(c)\times M_2^q(-c)
\]
is conformally flat, we can apply Theorem \ref{Thm-1}.
Hence, for a local orthonormal tangent frame $\{e_{j}\}_{1\leq j\leq n}$ on
$N^n$, using (\ref{Ricpro})
\[
\sum_{j=1}^n \widetilde{\operatorname{Ric}}(e_j,e_j)
=
(p-1)c\sum_{j=1}^n\|(\pi_1)_*e_j\|^2
-
(q-1)c\sum_{j=1}^n\|(\pi_2)_*e_j\|^2.
\]
Moreover, since the metric on $\widetilde{M}^{p+q}$ is the product metric,
we have
\[
\|e_j\|^2
=
\|(\pi_1)_*e_j\|^2+\|(\pi_2)_*e_j\|^2=1,
\]
and consequently
\[
\sum_{j=1}^n\|(\pi_2)_*e_j\|^2=n-\lambda.
\]
Therefore,
\[
\sum_{j=1}^n \widetilde{\operatorname{Ric}}(e_j,e_j)
=
c\left[(p-1)\lambda-(q-1)(n-\lambda)\right].
\]
Substituting this relation and the scalar curvature (\ref{scalpro}) into Theorem \ref{Thm-1}, with $m=p+q$, yields
\[
\rho+\rho^\perp
\leq
\|H\|^2
+
\frac{2c}{n(p+q-2)}
\left[(p-1)\lambda-(q-1)(n-\lambda)\right]
-
\frac{c(p-q)}{p+q-2}.
\]
\end{proof}

Using the above proposition, we have the following non-existence results for minimal submanifolds:

\begin{corollary}
Under the hypotheses of Proposition \ref{P-2}, if at a point $x\in N^n$,
\[
\rho+\rho^\perp
>
\frac{2c}{n(p+q-2)}
\left[(p-1)\lambda-(q-1)(n-\lambda)\right]
-
\frac{c(p-q)}{p+q-2},
\]
then $N^n$ is not minimal at $x$.
\end{corollary}

\begin{corollary}
Under the hypotheses of Proposition \ref{P-2}, we suppose that along $N^n$,
\[
\frac{2c}{n}
\left[(p-1)\lambda-(q-1)(n-\lambda)\right]
-
c(p-q)
\leq 0.
\]
If
\[
\rho+\rho^\perp>0
\]
at some point of $N^n$, then $N^n$ is not minimal at that point. In particular,
there exists no minimal immersion satisfying the above curvature condition
and $\rho+\rho^\perp>0$ everywhere.
\end{corollary}

\section{\textbf{Conclusion}}

An important problem in the theory of submanifolds is to provide a
relationship between the intrinsic and the extrinsic invariants. Generalized
Wintgen inequality establishes a relationship between intrinsic and
extrinsic curvatures and the squared mean curvature of a submanifold.
Continuing the present studies, we have proven the corresponding generalized
Wintgen inequality for submanifolds in conformally flat manifolds. This
study can be continued, and other generalized Wintgen inequality for submanifolds
in different types of manifolds can be obtained.

\small{
\section*{Statements and Declarations}

\textbf{Funding.} No financial support was received for this manuscript.

\noindent\textbf{Author Contribution.} C.\" O. and A.M.B. wrote and approved
the final version of the manuscript.

\noindent\textbf{Conflict of Interest.} The authors declare that there is no
conflict of interests.

\noindent\textbf{Data Availability Statement.} Not applicable to the current
study.

\bigskip

\noindent {\bf Cihan \"{O}zg\"{u}r}\\
Department of Mathematics, Faculty of Science and Letters, \.Izmir Democracy University\\
\.Izmir, T\"URK\.IYE\\
ORCID ID: 0000-0002-4579-7151\\
e-mail: cihan.ozgur@idu.edu.tr

\bigskip

\noindent {\bf Adara M. Blaga} (corresponding author)\\
Department of Mathematics, Faculty of Physics and Mathematics, 
West University of Timi\c soara\\
Timi\c soara, ROMANIA\\
ORCID ID: 0000-0003-0237-3866\\
e-mail: adarablaga@yahoo.com
}
\end{document}